\documentclass[11pt]{amsart}
\usepackage[margin=1.35in]{geometry}
\usepackage{amscd,amsmath,amsxtra,amsthm,amssymb,stmaryrd,xr,mathrsfs,mathtools,enumerate,commath, comment}

\usepackage{stmaryrd}
\usepackage[dvipsnames]{xcolor}
\definecolor{darkgoldenrod}{rgb}{0.72, 0.53, 0.04}
\definecolor{vegasgold}{rgb}{0.77, 0.7, 0.35}
\definecolor{gold(metallic)}{rgb}{0.83, 0.69, 0.22}
\definecolor{sepia}{rgb}{0.44, 0.26, 0.08}
\definecolor{silver}{rgb}{0.75, 0.75, 0.75}
\usepackage{commath}
\usepackage{comment}
\usepackage{tikz-cd}
\usepackage{longtable} % for 'longtable' environment
\usepackage{pdflscape} % for 'landscape' environment
\usepackage{booktabs}
\usepackage{hyperref}
\hypersetup{
 colorlinks=true,
 linkcolor=sepia,
 filecolor=vegasgold,
 citecolor=darkgoldenrod,
 urlcolor=sepia,
 pdftitle={Points on curves in inf. towers},
    }
\usepackage{enumerate}
\usepackage[utf8]{inputenc}
\usepackage[OT2,T1]{fontenc}
\usepackage{color}

\DeclareSymbolFont{cyrletters}{OT2}{wncyr}{m}{n}
\DeclareMathSymbol{\Sha}{\mathalpha}{cyrletters}{"58}

\numberwithin{equation}{section}
 
\usepackage{color}
\DeclareSymbolFont{cyrletters}{OT2}{wncyr}{m}{n}
\DeclareMathSymbol{\Sha}{\mathalpha}{cyrletters}{"58}

\newcommand{\Z}{\mathbb{Z}}

\newcommand{\Q}{\mathbb{Q}}

\newcommand{\op}[1]{\operatorname{#1}}
\newcommand{\F}{\mathbb{F}}

\newcommand{\Zx}{\mathbb{Z}_p\llbracket x\rrbracket}

\theoremstyle{plain}
 \theoremstyle{definition}
\newtheorem{Th}{Theorem}[section]
\newtheorem{Lemma}[Th]{Lemma}

\newtheorem{Corollary}[Th]{Corollary}

\newtheorem{Remark}[Th]{Remark}
 \theoremstyle{definition}
\newtheorem{Definition}[Th]{Definition}

\begin{document}

\title[Points on algebraic curves in infinite towers]{Rational points on algebraic curves in infinite towers of number fields}

\author[A.~Ray]{Anwesh Ray}
\address{Department of Mathematics\\
University of British Columbia\\
Vancouver BC, Canada V6T 1Z2}
\email{anweshray@math.ubc.ca}

\subjclass[2010]{11G10, 11G30, 11R23 (primary), 14H40 (secondary)}
\keywords{Iwasawa theory, Rational points of algebraic curves}

\begin{abstract} We study a natural question in the Iwasawa theory of algebraic curves of genus $>1$. Fix a prime number $p$. Let $X$ be a smooth, projective, geometrically irreducible curve defined over a number field $K$ of genus $g>1$, such that the Jacobian of $X$ has good ordinary reduction at the primes above $p$. Fix an odd prime $p$ and for any integer $n>1$, let $K_n^{(p)}$ denote the degree-$p^n$ extension of $K$ contained in $K(\mu_{p^{\infty}})$. We prove explicit results for the growth of $\#X(K_n^{(p)})$ as $n\rightarrow \infty$. When the Jacobian of $X$ has rank zero and the associated adelic Galois representation has big image, we prove an explicit condition under which $X(K_{n}^{(p)})=X(K)$ for all $n$. This condition is illustrated through examples. We also prove a generalization of Imai's theorem that applies to abelian varieties over arbitrary pro-$p$ extensions.
\end{abstract}

\maketitle
\section{Introduction}
\par Given an elliptic curve $E$ defined over a number field $K$, the Mordell--Weil group $E(K)$ is a finitely generated abelian group. The Iwasawa theory of elliptic curves was initiated by Mazur in \cite{mazur1972rational}, who studied the growth of the Mordell--Weil ranks of an elliptic curve in certain infinite towers of number fields. Given a prime $p$, let $K_n^{(p)}$ denote the Galois extension of $K$ with $[K_n^{(p)}:K]=p^n$ which is contained in the infinite cyclotomic extension $K(\mu_{p^\infty})$. It was shown by Kato \cite{kato2004p} and Rohlrich \cite{rohrlich1984onl} that if $E$ is defined over $\Q$ and $K$ is any abelian number field, then, $\op{rank} E(K_n^{(p)})$ is bounded as $n\rightarrow \infty$. The study of such questions led to the development of the Iwasawa theory of elliptic curves, and it was in this context that Mazur studied \emph{control theorems} for Selmer groups. With the development of non-abelian Iwasawa theory, such results have subsequently been extended to more general $p$-adic Lie-extensions, in which the asymptotic growth of Mordell--Weil ranks is studied, see for instance, \cite{harris1979systematic,lei2020ranks, delbourgo2017estimating, HL, ray2021asymptotic}. The study of growth questions in towers is an active area of research and is intimately related to Iwasawa theoretic properties of Selmer groups.

\par Given the interest in studying questions about the growth of ranks of elliptic curves, it is natural to ask similar questions for curves of higher genus. A curve $X_{/K}$ is said to be \emph{nice} if it is smooth, projective and geometrically irreducible. In this context, the celebrated theorem of Faltings \cite{faltings1983finity} states that if $X_{/K}$ is an nice curve of genus $g>1$, then the set of rational points $X(K_n^{(p)})$ is finite for all $n$. A natural question is to characterize the growth of $\# X(K_n^{(p)})$ as $n\rightarrow \infty$. It follows from results of Mazur \cite{mazur1972rational} and Imai \cite{imai1975remark} that $\# X(K_n^{(p)})$ is bounded provided some additional conditions are satisfied (see Theorem \ref{main th}). This requires that $p\geq 5$, that the Jacobian of $X$ has good ordinary reduction at all the primes of $K$ that lie above $p$, and the $p$-primary Selmer group of the Jacobian (over $K_{\op{cyc}}^{(p)}$) is cotorsion over the Iwasawa algebra. The last condition is conjectured by Mazur in \emph{loc. cit.}, and is known in the case when the Mordell--Weil rank of the Jacobian over $K$ is zero, and its Tate--Shafarevich group over $K$ is finite (see Remark \ref{cotorsion remark} and Theorem \ref{cotorsion zero}). This latter is verified in an explicit example, see Example \ref{example}.

\par We prove more explicit results regarding the growth of points in the cyclotomic tower. We prove an upper bound for the minimum value $m_0=m_0(p)$ such that $X(K_n^{(p)})=X(K_{m_0}^{(p)})$ for all $n>m_0$. Using this bound, we prove results about the variation of $m_0(p)$ as $p\rightarrow \infty$. This relies on explicitly understanding the torsion in the Jacobian of $X$ over the cyclotomic $\Z_p$-extension. We then consider the case when $\op{Jac}(X)$ has Mordell--Weil rank zero and establish an explicit criterion for $X(K_n^{(p)})$ to be equal to $X(K)$ for all $n$, see Theorem \ref{main 2}. This criterion indicates that for $100\%$ of primes $p$ above which the Jacobian of $X$ has good ordinary reduction, $X(K_{\op{cyc}}^{(p)})=X(K)$, in the rank zero setting.
\par We also obtain a generalization of Imai's theorem for all pro-$p$ extensions, see Theorem \ref{generalization of Imai}. The result states that if $A_{/K}$ is an abelian variety for which the associated adelic Galois representation has big image, and the $p$-torsion subgroup of $A(K)$ is trivial, then the torsion subgroup of $A(K_\infty)$ is finite, for all infinite pro-$p$ extensions $K_\infty/K$. This result makes no assumption on the reduction-type of $A$ at the primes above $p$. According to a theorem of Serre and Pink (see Theorem \ref{serre pink}), the big image assumption is satisfied for a large number of abelian varieties. In addition to this, we prove a boundedness result for the torsion subgroup of $A(K_\infty)$ for all pro-$p$ extensions $K_\infty/K$, as $p\rightarrow \infty$, see Theorem \ref{boundedness}. We also have an explicit example to illustrate this, see Example \ref{example}. Specializing to the case when $A$ is the Jacobian of a curve $X$, the result implies that if the rank of $A$ is bounded in the number fields contained in $K_\infty$, then $X(K_\infty)$ is finite.

\par\emph{Organization:} Including the introduction, the paper consists of four sections. In section \ref{section: preliminaries}, we introduce preliminary notions from Iwasawa theory. In section \ref{s 3}, we prove results about the growth of rational points in cyclotomic towers, namely Theorems \ref{main th}, \ref{th 2}, \ref{generalization of Imai} and \ref{boundedness}. In section \ref{s 4}, we establish a criterion for $X(K)=X(K_n^{(p)})$ for all $n$ in the setting when $\op{Jac}(X)$ has rank zero over $K$, and the adelic Galois representation has big image, see Theorem \ref{main 2}. 

\subsection*{Acknowledgements} The author would like to thank Jeffrey Hatley, Antonio Lei, Larry Washington and Tom Weston for helpful comments. The author is especially grateful to Ananth N.~Shankar for insightful suggestions including the idea used in the proof of Theorem \ref{main th}. The author would also like to thank the anonymous referee for careful and timely reading of the manuscript.
\section{Preliminaries}
\label{section: preliminaries}

Throughout, $p$ will be a prime number such that $p\geq 5$ and $K$ a number field. Let $X$ be a curve over $K$ of genus $g>1$. Assume that $X$ is \emph{nice} i.e., smooth, projective and irreducible over $\bar{K}$. Let $A$ denote the Jacobian of $X$ and assume that $A$ has good ordinary reduction at all primes above $p$.

\subsection{Galois representations and monodromy groups}

\par We introduce the \emph{adelic Galois representation} $\widehat{\rho}$ attached to $A$. For $m\in \Z_{\geq 1}$, we let $A[m]$ be the $m$-torsion subgroup of $A(\bar{K})$. At a prime $\ell$, the $\ell$-adic Tate-module is the inverse limit with respect to multiplication by $\ell$ maps
\[\op{T}_\ell(A):=\varprojlim_n A[\ell^n].\] On the other hand, the \emph{big Tate-module} is the inverse limit 
\[\op{T}(A):=\varprojlim_m A[m]=\prod_\ell \op{T}_\ell(A).\]The Galois action on $A[m]$ coincides with a representation \[\rho_{m}:\op{Gal}(\bar{K}/K)\rightarrow \op{GSp}_{2g}(\Z/m\Z).\] Let $\widehat{\rho}:\op{Gal}(\bar{K}/K)\rightarrow \op{GSp}_{2g}(\widehat{\Z})$ be the Galois representation on $\op{T}(A)$, and the image of $\widehat{\rho}$ is denoted $G\subset \op{GSp}_{2g}(\widehat{\Z})$. On the other hand the image of the $\ell$-adic Galois representation $\rho_{\ell^\infty}: \op{Gal}(\bar{K}/K)\rightarrow \op{GSp}_{2g}(\Z_\ell)$ is denoted $G_\ell$. The group $G$ (resp. $G_\ell$) is referred to as the \emph{big} (resp. \emph{$\ell$-adic}) \emph{monodromy group}. Note that $\rho_\ell$ is the mod-$\ell$ reduction of $\rho_{\ell^\infty}$, denote its image by $\bar{G}_\ell$.
\begin{Definition}\label{big image def}
The representation $\widehat{\rho}$ is said to have \emph{big image} if $G\cap \op{Sp}_{2g}(\widehat{\Z})$ is a finite index subgroup of $\op{Sp}_{2g}(\widehat{\Z})$.
\end{Definition}
\begin{Remark}
We do \emph{not} impose the big image condition in the main result, however, the hypothesis is used in proving a refined result in the special case when the Jacobian has rank zero, see Theorem \ref{main 2}.
\end{Remark}
The following result is due to Serre \cite{serre2000letter, serre2000resume} and Pink \cite{pink1998}.
\begin{Th}[Serre, Pink]\label{serre pink}
Let $A_{/\Q}$ be an abelian variety and assume that $\op{End}(A_{/\bar{\Q}})=\Z$. Then the image of $\widehat{\rho}$ contains a finite index subgroup of $\op{GSp}_{2g}(\widehat{\Z})$ provided $g=1,2$, or $g\geq 3$ is not in the set
\[\left\{\frac{1}{2}(2n)^k\mid n>0, k\geq 3 \text{ is odd}\right\}\cup\left\{\frac{1}{2}{{2n}\choose{n}}\mid n\geq 3 \text{ is odd}\right\} =\{4,10,16,32,\dots\}.\]
\end{Th}
Note that $G$ is closed, and if $\widehat{\rho}$ has big image, then $G':=G\cap \op{Sp}_{2g}(\widehat{\Z})$ is an open subgroup of $\op{Sp}_{2g}(\widehat{\Z})$. As a result, $G_\ell$ contains $\op{Sp}_{2g}(\Z_\ell)$ for all but finitely many primes $\ell$.

\subsection{Selmer groups and number field towers}

\par The results proven in this subsection shall apply to all abelian varieties $A$ with good ordinary reduction at the primes of $K$ that lie above $p$, however the case of interest is when $A$ is the Jacobian of the curve $X$ from the previous subsection. We introduce the main object of interest in Iwasawa theory, namely, the $p$-primary Selmer group over the cyclotomic $\Z_p$-extension. For $n\in \Z_{\geq 1}$ set $\Q_n^{(p)}$ to be the unique extension of $\Q$ contained in $\Q(\mu_{p^{n+1}})$ such that $[\Q_n^{(p)}:\Q]=p^n$. Let $\Q_{\op{cyc}}^{(p)}$ be the infinite extension obtained by taking the union and set $K_{\op{cyc}}^{(p)}:=K\cdot \Q_{\op{cyc}}^{(p)}$. It is easy to show that $\Gamma:=\op{Gal}(K_{\op{cyc}}^{(p)}/K)$ is isomorphic to $\Z_p$. For $n\in \Z_{\geq 0}$, let $K_n^{(p)}$ be the extension of $K$ contained in $K_{\op{cyc}}^{(p)}$ such that $[K_n^{(p)}:K]=p^n$. The tower of number fields
\[K=K_0^{(p)}\subseteq K_1^{(p)}\subseteq K_2^{(p)}\subseteq \dots \subseteq K_n^{(p)}\subseteq \dots\] is called the cyclotomic tower.
\par Throughout, we pick a topological generator $\gamma$ of $\Gamma$. The Iwasawa algebra $\Lambda$ is defined to be the inverse limit of finite group rings 
\[\Lambda:=\varprojlim_n \Z_p[\op{Gal}(K_n^{(p)}/K)].\] There is an isomorphism of $\Lambda$ with the formal power series ring $\Z_p\llbracket x\rrbracket $ which takes $\gamma$ to $x+1$. From here on in, we choose such an isomorphism and identify $\Lambda$ with the formal power series ring $\Zx$.
\par Given a number field $F/K$, let $\op{Sel}_{p^\infty}(A/F)$ be the $p$-primary Selmer group of $A$ over $F$. We recall the definition here. Let $S$ be the set of primes $v$ of $F$ such that either $v|p$, or, $A$ has bad reduction at $v$. Let $F_S$ be the maximal algebraic extension of $F$ in which all primes $v\nmid S$ are unramified. Denote by $F_v$ the completion of $F$ at $v$. The Kummer sequence 
\[0\rightarrow A[p^n]\rightarrow A\rightarrow A\rightarrow 0\] induces a map in cohomology 
\[H^1(F_v, A[p^n])\rightarrow H^1(F_v, A)[p^n].\] Passing to the direct limit as $n\rightarrow \infty$, we have the following map
\[H^1(F_v, A[p^\infty])\rightarrow H^1(F_v, A)[p^\infty].\]
The Selmer group is defined as follows
\[\op{Sel}_{p^\infty}(A/F):=\op{ker} \left\{H^1(F_S/F, A[p^\infty])\longrightarrow \bigoplus_{v\in S} H^1(F_v, A)[p^\infty]\right\},\] where the above restriction map factors through $\bigoplus_{v\in S} H^1(F_v, A[p^\infty])$.

\par Let $\Sha(A/F)$ be the \emph{Tate-Shafarevich group}, defined as follows
\[\Sha(A/F):=\left\{H^1(\bar{F}/F, A(\bar{F}))\rightarrow \prod_{l} H^1(\bar{F}_v/F_v, A(\bar{F}_v))\right\}.\]The Selmer group $\op{Sel}_{p^\infty}(A/F)$ fits into a short exact sequence 
\begin{equation}\label{selmer ses}0\rightarrow A(F)\otimes \Q_p/\Z_p\rightarrow \op{Sel}_{p^\infty}(A/F)\rightarrow \Sha(E/F)[p^\infty]\rightarrow 0.\end{equation}
\par The Selmer group over $K_{\op{cyc}}^{(p)}$ is taken as the direct limit over restriction maps
\[\op{Sel}_{p^\infty}(A/K_{\op{cyc}}^{(p)}):=\varinjlim_n \op{Sel}_{p^\infty}(A/K_n^{(p)}).\]

It is well known that $\op{Sel}_{p^\infty}(A/K_{\op{cyc}}^{(p)})$ is a cofinitely generated module over the Iwasawa-algebra $\Zx$, i.e., its Pontryagin dual 
\[\op{Sel}_{p^\infty}(A/K_{\op{cyc}}^{(p)})^{\vee}:=\op{Hom}\left(\op{Sel}_{p^\infty}(A/K_{\op{cyc}}^{(p)}), \Q_p/\Z_p\right)\]is finitely generated over $\Zx$.

\begin{Th}\label{cotorsion zero}
Suppose that $\op{Sel}_{p^\infty}(A/K)$ is finite, i.e., $\op{rank} A(K)=0$ and $\Sha(E/K)[p^\infty]$ is finite. Then, the Selmer group $\op{Sel}_{p^\infty}(A/K_{\op{cyc}}^{(p)})$ is cotorsion over $\Zx$.
\end{Th}

\begin{proof}
The statement follows from results in \cite{coates1996kummer} and the proof of \cite[Theorem 2.8]{coates2000galois}. In fact, the argument in the proof of \cite[Theorem 2.8]{coates2000galois} generalizes verbatim to abelian varieties with good ordinary reduction at primes above $p$.
\end{proof}

\begin{Remark}\label{cotorsion remark}
More generally, it was conjectured by Mazur in \cite{mazur1972rational} that the Selmer group $\op{Sel}_{p^\infty}(A/K_{\op{cyc}}^{(p)})$ is always cotorsion over $\Zx$.
\end{Remark}

\subsection{}
We now introduce Iwasawa-invariants associated to $\Zx$-modules. Let $M$ be a cofinitely generated cotorsion $\Z_p\llbracket x\rrbracket$-module, i.e., the Pontryagin-dual \[M^{\vee}:=\op{Hom}(M, \Q_p/\Z_p)\] is a finitely generated and torsion $\Z_p\llbracket x\rrbracket$-module. Recall that a polynomial $f(x)\in \Zx$ is said to be \textit{distinguished} if it is a monic polynomial whose non-leading coefficients are all divisible by $p$. Note that all height $1$ prime ideals of $\Z_p\llbracket x\rrbracket$ are principal ideals $(a)$, where $a=p$ or $a=f(x)$, where $f(x)$ is an irreducible distinguished polynomial. 
According to the structure theorem for $\Zx$-modules (see \cite[Theorem 13.12]{washington1997introduction}), $M^{\vee}$ is pseudo-isomorphic to a finite direct sum of cyclic $\Zx$-modules, i.e., there is a map
\[
M^{\vee}\longrightarrow \left(\bigoplus_{i=1}^s \Zx/(p^{\mu_i})\right)\oplus \left(\bigoplus_{j=1}^t \Zx/(f_j(x)) \right)
\]
with finite kernel and cokernel.
Here, $\mu_i>0$ and $f_j(x)$ is a distinguished polynomial.
The \emph{characteristic ideal} of $M^\vee$ is generated by the element
\begin{equation}\label{char element}
f_{M}^{(p)}(x) = f_{M}(x) := p^{\sum_{i} \mu_i} \prod_j f_j(x),
\end{equation}
which is referred to as the \emph{characteristic element} associated to $M$.
The $\mu$-invariant of $M$ is defined as the power of $p$ in $f_{M}(x)$.
More precisely,
\[
\mu(M):=\begin{cases}
\sum_{i=1}^s \mu_i & \textrm{ if } s>0\\
0 & \textrm{ if } s=0.
\end{cases}
\]
The $\lambda$-invariant of $M$ is the degree of the characteristic element, i.e.,
\[
\lambda(M) :=\begin{cases}
\sum_{i=1}^s \deg f_i & \textrm{ if } s>0\\
0 & \textrm{ if } s=0.
\end{cases}
\]

Assuming that $\op{Sel}_{p^\infty}(A/K_{\op{cyc}}^{(p)})$ is cotorsion as a $\Zx$-module, set $\mu(A/K_{\op{cyc}}^{(p)})$ and $\lambda(A/K_{\op{cyc}}^{(p)})$ to be the $\mu$ and $\lambda$-invariants of the Selmer group $\op{Sel}_{p^\infty}(A/K_{\op{cyc}}^{(p)})$ respectively.

\section{Bounding the growth of rational points in the cyclotomic tower}\label{s 3}
\par Recall that $X$ is a nice curve over $K$ and $A$ is the Jacobian of $A$. Let $p$ be an odd prime. We assume throughout this section that
\begin{itemize}
    \item $A$ has good ordinary reduction at the primes above $p$,
    \item $\op{Sel}_{p^\infty}(A/K_{\op{cyc}}^{(p)})$ is cotorsion over $\Zx$ (see Theorem \ref{cotorsion zero} and Remark \ref{cotorsion remark}).
\end{itemize}

\subsection{Mordell's conjecture over $\Z_p$-extensions}

\par First, we recall the well--known result of Mazur, see \cite{mazur1972rational}.
\begin{Th}[Mazur]\label{mazur's theorem}
There exists $n_0$ such that $\op{rank} A(K_{n}^{(p)})=\op{rank} A(K_{n_0}^{(p)})$ for all $n>n_0$. Furthermore, $\op{rank} A(K_{n}^{(p)})$ is bounded above by the $\lambda$-invariant $\lambda(A/K_{\op{cyc}}^{(p)})$.
\end{Th}
In fact, Mazur showed that the entire Mordell--Weil group $A(K_{\op{cyc}}^{(p)})$ is finitely generated when $A$ is either an elliptic curve or an abelian variety with complex--multiplication. The above result crucially requires that the assumptions on $A$ made at the start of this section.
\begin{Definition}\label{def n0}
Denote by $n_0=n_0(p)$ the minimum value such that \[\op{rank} A(K_{n}^{(p)})=\op{rank} A(K_{n_0}^{(p)})\] for all $n>n_0$.
\end{Definition} Note that according to Mazur's result, if $\op{rank} A(K)=\lambda(A/K_{\op{cyc}}^{(p)})$, then, $n_0(p)=0$, i.e.,
 \[\op{rank}A(K_{\op{cyc}}^{(p)})=\op{rank}A(K).\]
 We also recall the result of Imai, see the main theorem and last few lines of \cite{imai1975remark}.
 \begin{Th}[Imai]
 Let $A$ be as above, then the torsion subgroup of $A(K_{\op{cyc}}^{(p)})$ is finite.
 \end{Th}

 For ease of notation, set $\alpha(p)$ to denote the order of the torsion subgroup of $A(K_{\op{cyc}}^{(p)})$.
 
\begin{Th}\label{main th}
Let $X$ be a nice curve of genus $g>1$ defined over a number field $K$ and $A$ let be the Jacobian of $X$. Assume that the following conditions hold:
\begin{enumerate}
    \item $A$ has good ordinary reduction at the primes above $p$,
    \item $\op{Sel}_{p^\infty}(A/K_{\op{cyc}}^{(p)})$ is cotorsion over $\Zx$.
\end{enumerate}
Then, $X\left(K_{\op{cyc}}^{(p)}\right)$ is finite. Furthermore, assume that $X(K)\neq \emptyset$, and let $m_0$ be the minimum integer such that $X(K_n^{(p)})=X(K_{m_0}^{(p)})$ for all $n>m_0$. Then, we have that $m_0\leq n_0+\lfloor \op{log}_p \alpha(p) \rfloor$, where $n_0$ is defined above (see Definition \ref{def n0}). 
\end{Th}
\begin{proof} Recall that $n_0$ is minimal such that
\[\op{rank} A(K_n^{(p)})=\op{rank} A(K_{n_0}^{(p)})\] for all $n>n_0$. Assume without loss of generality that for some $n>n_0$, 
\[A(K_n^{(p)})\neq A(K_{n_0}^{(p)}),\] and let $Q\in A(K_n^{(p)})$ be such that $Q\notin A(K_{n_0}^{(p)})$. Let $n'$ be such that $Q\in A(K_{n'}^{(p)})$ and $Q\notin A(K_{n'-1})$, we have that $n_0<n' \leq n$. Since \[\op{rank} A(K_{n'}^{(p)})=\op{rank} A(K_{n_0}^{(p)})\]
it follows that there is an integer $N>0$ such that $NQ\in A(K_{n_0}^{(p)})$.
\par For $\sigma\in \op{Gal}(K_{n'}^{(p)}/K_{n_0}^{(p)})$, the point $P_\sigma:=Q-\sigma(Q)$ is an $N$-torsion point in $A(K_{\op{cyc}}^{(p)})$. Moreover, as $\sigma$ ranges over $\op{Gal}(K_{n'}^{(p)}/K_{n_0}^{(p)})$ the points $\sigma(Q)$ are all distinct, since it is assumed that $Q\notin A(K_{n'-1})$. Hence, the points $P_\sigma$ are all distinct points in $A(K_{\op{cyc}}^{(p)})_{\op{tors}}$. According to Imai's theorem, $A(K_{\op{cyc}}^{(p)})_{\op{tors}}$ is finite. The number of points $P_\sigma$ produced is $[K_{n'}^{(p)}:K_{n_0}^{(p)}]=p^{n'-n_0}$. 

Hence, we find that $p^{n'-n_0}\leq \alpha(p)$, i.e., 
\[n'\leq n_0+\lfloor\log_p\alpha(p)\rfloor.\] 
Therefore, this shows that $A(K_{\op{cyc}}^{(p)})=A(K_{m_0}^{(p)})$, where $m_0\leq n_0+\lfloor\log_p\alpha(p)\rfloor$.
\par Assume first that $X(K)\neq \emptyset$. Given a point $P_0\in X(K)$, consider the standard embedding $X\hookrightarrow A$ mapping a point $P$ to $[P]:=P_0-P$. Suppose that $P\in X(K_n^{(p)})$, then, we know that $[P]=P_0-P\in A(K_{m_0}^{(p)})$. Hence, $P$ is defined over $K_{m_0}^{(p)}$, this shows that $X(K_n^{(p)})\subseteq X(K_{m_0}^{(p)})$, and hence, 
\[X(K_{\op{cyc}}^{(p)})=X(K_{m_0}^{(p)}).\]
\par It remains to show that $X\left(K_{\op{cyc}}^{(p)}\right)$ is finite when $X(K)=\emptyset $. Without loss of generality, $X\left(K_{\op{cyc}}^{(p)}\right)\neq \emptyset$, and thus, $X(K_N^{(p)})\neq \emptyset $ for some value of $N$. Replace $K$ with $K_N^{(p)}$ and the above argument shows that $X\left(K_{\op{cyc}}^{(p)}\right)$ is finite.
\end{proof}
Thus the Mordell conjecture over cyclotomic $\Z_p$-extensions follows from results of Mazur and Imai. However, the above result is more explicit, as it gives more precise information about growth. We note here that $A$ is an abelian variety of $\op{GL}_2$-type with supersingular reduction at the primes above $p$, it is shown by Lei and Ponsinet (see \cite{lei2020mordell}) that the Mordell--Weil group is stable in the cyclotomic tower.

\subsection{Torsion over pro-$p$ extensions}

\par One would like to better understand the quantity $m_0$ from the statement of Theorem \ref{main th}. We shall prove in section \ref{s 4} that $m_0=0$ under some additional conditions. The result of Imai shows that the torsion subgroup of $A(K_{\op{cyc}}^{(p)})$ is finite. In this section, we shall obtain more precise information about this group. As a result, we show that if $\widehat{\rho}$ has big image, then, $\alpha(p)$ is bounded as $p\rightarrow \infty$. Recall that $\bar{G}_\ell$ is the image of the mod-$\ell$ residual representation on $A[\ell]$. 
\par Throughout, consider a Galois extension $K_\infty$ of $K$ and assume that the Galois group $G=\op{Gal}(K_\infty/K)$ is pro-$p$. The lower central $p$-series of $G$ is recursively defined as follows: \[G_0:=G\text{ and }G_{n+1}:=G_n^p[G_n, G].\]
Set $K_n$ to be $K^{G_n}$, observe that $\op{Gal}(K_{n+1}/K_n)\simeq \left(\Z/p\Z\right)^{d_n}$ for some integer $d_n$.

\par The following result shows that if $\bar{G}_\ell$ is suitably large, then $A(K_\infty)[\ell^\infty]=0$.

\begin{Lemma}\label{lem 1}
Let $\ell$ be a prime such that
\begin{enumerate}
    \item $\bar{H}_\ell:=\bar{G}_\ell\cap \op{Sp}_{2g}(\F_\ell)$ does not have any (non-trivial) abelian $p$-quotients, 
    \item $A[\ell]^{\bar{H}_\ell}=0$,
\end{enumerate}
then, $A(K_\infty)[\ell^\infty]=0$. 
\end{Lemma}
\begin{proof}
Let $\rho_\ell:\op{Gal}(\bar{K}/K)\rightarrow \bar{G}_\ell$ denote the mod-$\ell$ residual representation and set 
\[U:=\rho_\ell\left(\op{Gal}(\bar{K}/K_\infty)\right)\subseteq \bar{G}_\ell.\] Since $\bar{H}_\ell$ has no $p$-abelian quotients, $\bar{H}_\ell\cap U=\bar{H}_\ell$. Thus, $U$ contains $\bar{H}_\ell$, and hence, 
\[A(K_\infty)[\ell]=A[\ell]^U=0.\] Since $A(K_\infty)[\ell]=0$, it follows that $A(K_\infty)[\ell^\infty]=0$.
\end{proof}

\begin{Corollary}\label{cor simple}
Let $\ell$ be a prime such that $\bar{G}_\ell$ contains $\op{Sp}_{2g}(\F_\ell)$. Assume that either $g>2$ or $\ell>2$, then, $A(K_\infty)[\ell^\infty]=0$.
\end{Corollary}
\begin{proof}
If either $g>2$ or $\ell>2$, then, $\op{PSp}_{2g}(\F_\ell)$ is simple (see \cite[Theorem 3.4.1]{o1978symplectic}). Since $p$ is odd, the result follows from Lemma \ref{lem 1}.
\end{proof}
\begin{Remark}
Suppose that $\widehat{\rho}$ has big image, then, $\bar{G}_\ell$ contains $\op{Sp}_{2g}(\F_\ell)$ for all but a finite exceptional set of primes $\Sigma$. The above result implies that $A(K_\infty)[\ell^\infty]=0$ outside the exceptional set.
\end{Remark}

The next result applies to all primes $\ell\neq p$ and shows that $\#A(K_\infty)[\ell^\infty]$
is bounded independent of $p$ and the extension $K_\infty$.
\begin{Lemma}\label{lem 2}
Let $\ell$ be any prime such that $\ell\neq p$ and $L$ be the field $K(A[\ell])$. Then, $A(K_\infty)[\ell^\infty]\subseteq A(L)[\ell^\infty]$.
\end{Lemma}
\begin{proof}
Let $L(m)$ be the field $K(A[\ell^m])$, which is the field fixed by the kernel of the Galois representation 
\[\varrho_m:=\rho_{\ell^m}: \op{Gal}(\bar{K}/K)\rightarrow \op{GSp}_{2g}(\Z/\ell^m\Z).\] Given an extension $\mathcal{K}/K$, set $H_{m,\mathcal{K}}$ to denote $\varrho_m\left(\op{Gal}(\bar{K}/\mathcal{K})\right)$. Note that by construction, \begin{equation}\label{boring equation}A(\mathcal{K})[\ell^m]=A[\ell^m]^{H_{m,\mathcal{K}}}.\end{equation}
Also note that since $L:=K(A[\ell])$, 
\[H_{m,L}=\left(\op{image}\varrho_m\right)\cap \mathcal{U}_m,\]
where 
\[\mathcal{U}_m:=\op{ker}\left(\op{GSp}_{2g}(\Z/\ell^m\Z)\rightarrow \op{GSp}_{2g}(\Z/\ell\Z)\right).\]
Observe that $\mathcal{U}_m$ is pro-$\ell$, and hence so is $H_{m,L}$. On the other, by construction $\varrho_m$ induces an isomorphism $\op{Gal}(L(m)/L)\simeq H_{m,L}$, hence, $L(m)/L$ is a pro-$\ell$ extension. On the other hand, $L_\infty:=L\cdot K_\infty$ is a pro-$p$ extension of $L$, and since $p\neq \ell$, we have that $L_\infty\cap L(m)=L$. Therefore, we have that 
\[H_{m,L}= \op{Gal}(L(m)/L)=\op{Gal}(L(m)/L_\infty\cap L(m))=H_{m,L_\infty},\] and as a consequence of \eqref{boring equation}, we deduce that 
\[A(L)[\ell^m]=A(L_\infty)[\ell^m],\] and thus, taking the union of all $\ell^m$-torsion,
\[A(L)[\ell^\infty]=A(L_\infty)[\ell^\infty].\]
This proves the result.
\end{proof}

\begin{Lemma}\label{lem 3}
Suppose that $A(K)[p^\infty]=0$, then, we have that $A(K_\infty)[p^\infty]=0$.
\end{Lemma}
\begin{proof}
For any value of $m>0$, it suffices to show that $A(K_{m})[p^\infty]=0$. Letting $M=A(K_{m})[p^\infty]$ and $\mathcal{G}_m=\op{Gal}(K_m/K)$, we have that $M^{\mathcal{G}_m}=0$. Since $\mathcal{G}_m$ is a finite $p$-group, it follows that $M=0$, see \cite[Proposition 1.6.12]{neukirch2013cohomology}. 
\end{proof}

We obtain the following corollary to Lemmas \ref{lem 1}, \ref{lem 2} and \ref{lem 3}. Given an abelian variety $A_{/K}$, we let $\Sigma=\Sigma_A$ be the finite set of primes $\ell$ such that $\bar{G}_\ell$ does \emph{not} contain $\op{Sp}_{2g}(\F_\ell)$, and given $\ell\in \Sigma$, let $K_{A,\ell}:=K(A[\ell])$.

\begin{Th}\label{alpha p bounded}
Let $A$ be any abelian variety over $K$ such that $\widehat{\rho}$ has big image. Given a prime $p$ at which $A(K)[p^\infty]=0$, we have the following bound on $\alpha(p):=\#A(K_{\op{cyc}}^{(p)})_{\op{tors}}$
\begin{equation}\label{new bound on  alpha}\alpha(p)\leq \prod_{\ell\in \Sigma} \#\left( A(K_{A,\ell})[\ell^\infty]\right).\end{equation}

\end{Th}

\begin{proof}
Since $\widehat{\rho}$ has big image, it follows from Corollary \ref{cor simple} that $A(K_{\op{cyc}}^{(p)})[\ell^\infty]=0$ for all primes $\ell\notin \Sigma\cup \{p\}$. Therefore, we find that 
\[\alpha(p)=\#A(K_{\op{cyc}}^{(p)})_{\op{tors}}\leq \prod_{\ell\in \Sigma\cup \{p\}} \#\left(A(K_{\op{cyc}}^{(p)})[\ell^\infty]\right).\]For $\ell\neq p$, it follows from Lemma \ref{lem 2} that $\#A(K_{\op{cyc}}^{(p)})[\ell^\infty]$ is bounded by $\# A(K_{A,\ell})[\ell^\infty]$. Since $A(K)[p^\infty]=0$ by assumption, it follows from Lemma \ref{lem 3} that $\#A(K_{\op{cyc}}^{(p)})[p^\infty]=0$. Putting everything together, we obtain the result.
\end{proof}
Clearly, \eqref{new bound on alpha} implies that $\alpha(p)$ is bounded as $p\rightarrow \infty$, which is a result originally due to Ribet (see \cite{ribet1981torsion}) and does not require the big image hypothesis. We shall generalize this result to arbitrary pro-$p$ extensions $K_\infty/K$, see Theorem \ref{boundedness}. The hypothesis is in place here so that we obtain the effective bound \eqref{new bound on  alpha}, which is useful in making our results effective.
We obtain the following result. 

\begin{Th}\label{th 2}
Let $X$ be as in Theorem \ref{main th}, $A$ its Jacobian and $n_0(p)$ as in Definition \ref{def n0}. Assume that $\widehat{\rho}$ has big image and that $X(K)\neq \emptyset$. Then, for $p\gg 0$ at which the conditions of Theorem \ref{main th} are satisfied,
\[X(K_{\op{cyc}}^{(p)})=X(K_{n_0(p)}).\]
In fact, we have a more effective statement. Recall that $\Sigma=\Sigma_A$ be the finite set of primes $\ell$ such that $\bar{G}_\ell$ does \emph{not} contain $\op{Sp}_{2g}(\F_\ell)$, and for $\ell\in \Sigma$, we have set $K_{A,\ell}:=K(A[\ell])$. Then, for all primes $p$ such that 
\begin{enumerate}
    \item $A(K)[p^\infty]=0$
    \item $p>\prod_{\ell\in \Sigma} \#\left( A(K_{A,\ell})[\ell^\infty]\right)$,
\end{enumerate} the above equality holds.
\end{Th}
\begin{proof}
The result follows from Theorem \ref{main th} and Theorem \ref{alpha p bounded}. Since $\alpha(p)$ is bounded as $p\rightarrow \infty$, it follows that $\lfloor \op{log}_p \alpha(p)\rfloor=0$ for $p\gg 0$. Since $n_0(p)\leq m_0(p)\leq n_0(p)+\lfloor \op{log}_p \alpha(p)\rfloor$, we find that $m_0(p)=n_0(p)$ for primes $p\gg 0$ (at which the conditions of Theorem \ref{main th} are satisfied). Moreover, if $A(K)[p^\infty]=0$, then it follows from Theorem \ref{alpha p bounded} that $\lfloor \op{log}_p \alpha(p)\rfloor=0$ for 
\[p>\prod_{\ell\in \Sigma} \#\left( A(K_{A,\ell})[\ell^\infty]\right).\]

Hence, if $A(K)[p^\infty]=0$ and \[p>\prod_{\ell\in \Sigma} \#\left( A(K_{A,\ell})[\ell^\infty]\right),\] we find that \[X(K_{\op{cyc}}^{(p)})=X(K_{n_0(p)}).\]
\end{proof}

\begin{Remark}
If $\op{rank} A(K)=\lambda(A/K_{\op{cyc}}^{(p)})$, then, $n_0(p)=0$. In the next section, we study the case when $\op{rank} A(K)=0$ and show that in this case, there is an explicit criterion for $\lambda(A/K_{\op{cyc}}^{(p)})=0$.
\end{Remark}

\subsection{A generalization of Imai's theorem}
\par Some of the results proven in the previous subsection generalize to arbitrary pro-$p$ extensions of $K$. Recall that $K_\infty$ is a pro-$p$ extension of $K$.
\begin{Th}\label{generalization of Imai}
Let $A$ be an abelian variety defined over $K$ such that the image of $\widehat{\rho}$ is large. Let $K_\infty$ be any pro-$p$ extension of $K$. Furthermore assume that $A(K)$ has no $p$-torsion. Then, the torsion subgroup of $A(K_\infty)$ is finite.
\end{Th}
\begin{proof}
The result follows from Corollary \ref{cor simple} and Lemmas \ref{lem 1}, \ref{lem 2} and \ref{lem 3}. In greater detail, since $\widehat{\rho}$ is assumed to have big image, we have that $\bar{G}_\ell$ contains $\op{Sp}_{2g}(\F_\ell)$ for all $\ell$ outside a finite set of primes $\Sigma$. Corollary \ref{cor simple} and Lemma \ref{lem 1} then imply that $A(K_\infty)[\ell^\infty]=0$ for all but a finite set of primes $\ell$. Lemma \ref{lem 2} asserts that $A(K_\infty)[\ell^\infty]$ is finite for all primes $\ell\neq p$. Finally, Lemma \ref{lem 3} shows that $A(K_\infty)[p^\infty]=0$. Putting everything together, we obtain the result.
\end{proof}
Note that the above result applies without any assumption on the reduction type of $A$ and hence is a partial improvement of Imai's result even for $\Z_p$-extensions, in the special case when $A(K)[p]=0$. In general, one cannot prove that the rank of an abelian variety stabilizes in an arbitary pro-$p$ extension, however, there are some special cases in which this is the case, see the results in \cite{HL}. When the rank of the Jacobian of a curve stabilizes in a general pro-$p$ extension, the above result and the argument in the proof of Theorem \ref{main th} shows that the number of points on the curve should be finite. Furthermore, we have the following result which shows that the order of the torsion subgroup of $A(K_\infty)$ is bounded independent of $p$ and the choice of pro-$p$ extension $K_\infty$.

\begin{Th}\label{boundedness}
Let $A$ be an abelian variety defined over a number field $K$ such that $\widehat{\rho}$ has big image. Then, there is a constant $C>0$ which depends only on $A$ such that for all primes $p\gg 0$ and all pro-$p$ extensions $K_\infty/K$, we have that 
\[\# A(K_\infty)_{\op{tors}}<C.\]
\end{Th}
\begin{proof}
Let $p\gg 0$ be such that $A(K)[p]=0$. Then, $A(K_\infty)[p^\infty]=0$ by Lemma \ref{lem 3}. Let $\Sigma$ be the finite set of primes outside which $\bar{G}_\ell$ contains $\op{Sp}_{2g}(\F_\ell)$. Then, by Lemma \ref{lem 1} and Corollary \ref{cor simple}, $A(K_\infty)[\ell^\infty]=0$. Finally, for each prime $\ell\in \Sigma$, it follows from Lemma \ref{lem 2} that $A(K_\infty)[\ell^\infty]\subseteq A\left(K(A[\ell])\right)$, and hence, its order is bounded independent of $p$ and $K_\infty$. The result follows.
\end{proof}
\section{The special case when the Jacobian has rank zero}\label{s 4}
\par In this section, we assume that $\op{rank} A(K)=0$. In this case, we show that there is an explicit criterion for $n_0(p)=0$. In other words, there is an explicit criterion for 
\[\op{rank} A(K_n^{(p)})=\op{rank} A(K)=0\] for all $n$. If $\widehat{\rho}$ has big image, then by Theorem \ref{th 2}, $m_0(p)=n_0(p)$ for primes $p\gg 0$ above which $A$ has good ordinary reduction. Hence, for $p\gg 0$, this gives an explicit criterion for $A(K_n)=A(K)$ for all $n$, and consequently, a criterion for $X(K_{\infty})$ to equal $X(K)$.
\subsection{Criterion for $X(K_{\op{cyc}}^{(p)})$ to equal $X(K)$}

\par Note that when $\op{rank} A(K)=0$, the Selmer group $\op{Sel}_{p^\infty}(A/K)$ isomorphic to the $p$-primary part of the Tate-Shafarevich group $\Sha(E/K)$. Assuming that $\Sha(E/K)[p^\infty]$ is finite, it follows from Theorem \ref{cotorsion zero} that the Selmer group $\op{Sel}_{p^\infty}(A/K_{\op{cyc}}^{(p)})$ is cotorsion over $\Zx$. Assume further that $A$ has good ordinary reduction at the primes above $p$.

\par According to Mazur's result, $\op{rank} A(K_n^{(p)})\leq \lambda(A/K_{\op{cyc}}^{(p)})$ for all $n>0$. Therefore, if $\lambda(A/K_{\op{cyc}}^{(p)})=0$, then, $\op{rank} A(K_n^{(p)})=0$ for all $n>0$. Let $f(x)$ be the characteristic element of $\op{Sel}_{p^\infty}(A/K_{\op{cyc}}^{(p)})$. Let $|\cdot|_p$ be the absolute value, normalized by $|p|_p^{-1}=p$. At a prime $v|p$, denote by $\F_v$ the residue field at $v$. Set $\widetilde{A}(\F_v)$ to be the group of $\F_v$-points of the reduced abelian variety $\widetilde{A}$ (at $v$). It follows from \cite[Theorem 2']{schneider1985p} (also see the proof of the Euler characteristic formula in \cite[section 3]{coates2000galois}) that $f(0)\neq 0$ and
\begin{equation}\label{ECF}f(0)= \frac{\#\Sha(A/K)[p^\infty]\times \prod_{v\nmid p} c_v^{(p)}(A/K)\times \left(\prod_{v\mid p}\#\widetilde{A}(\F_v)[p^\infty]\right)^2 }{\left(\#A(K)[p^\infty]\right)^2}.\end{equation}

 Here, $\prod_{v\nmid p}c_v(A/K)$ is the Tamagawa number product of $A$ taken at the primes $v\nmid p$, and $c_v^{(p)}(A/K)$ is the $p$-primary part of the Tamagawa number $c_v(A/K)$, defined by
\[c_v^{(p)}(A/K):=|c_v(A/K)|_p^{-1}.\] Note that the formula above is slightly different from that of \emph{loc. cit.} since in this context, $A$ being the Jacobian of a curve, is isomorphic to its own dual.

\begin{Definition}
Given an abelian variety $A_{/K}$ and a prime $p$ above which $A$ has good reduction, we say that $p$ is \emph{anomalous} if $p\mid \#\widetilde{A}(\F_v)$ for some prime $v\mid p$.
\end{Definition}
For elliptic curves $E$ defined over a number field $K$, the set of anomalous primes makes up $0\%$ of all primes, see \cite{murty1997modular}. We are not aware if a generalization of this result exists in the literature for abelian varieties of dimension $>1$, however, a related question is studied in \cite{cojocaru2017arithmetic, bloom2018square}.

\begin{Lemma}\label{lambda is zero}
Let $A$ be an abelian variety defined over a number field $K$ and assume that
\begin{enumerate}
    \item $\Sha(A/K)$ is finite,
    \item $\op{rank} A(K)=0$.
\end{enumerate} Then, for all but finitely many non-anomalous primes $p$ above which $A$ has good ordinary reduction, we have that $\lambda(A/K_{\op{cyc}}^{(p)})=0$. More precisely, if $p$ is a non-anomalous prime such that
\begin{enumerate}
    \item $A$ has good ordinary reduction at $p$,
    \item $\Sha(A/K)[p^\infty]=0$
    \item $p\nmid \prod_{v\nmid p} c_v(A/K)$,
    \end{enumerate}
    then, $\lambda(A/K_{\op{cyc}}^{(p)})=0$.

\end{Lemma}

\begin{proof}
Note that $f(x)$ is a power of $p$ times a distinguished polynomial, see \eqref{char element}. Therefore, if $f(0)$ is a unit in $\Z_p$, then $f(x)=1$. Note that $\lambda(A/K_{\op{cyc}}^{(p)})$ is the degree of this distinguished polynomial, and hence, if $f(0)\in \Z_p^\times$, then $\lambda(A/K_{\op{cyc}}^{(p)})=0$. It is easy to see from \eqref{ECF} that if $p$ satisfies the above conditions, then $f(0)$ is a unit in $\Z_p$.
\end{proof}

\begin{Th}\label{main 2}
Let $X$ be a nice curve of genus $g>1$ defined over a number field $K$ and $A$ be the Jacobian of $X$. Assume that the following conditions hold:
\begin{enumerate}
    \item $\op{rank} A(K)=0$,
    \item $\Sha(A/K)$ is finite,
    \item $\widehat{\rho}$ has big image in the sense of Definition \ref{big image def},
    \item $X(K)\neq \emptyset$.
\end{enumerate}
Then, for all non-anomalous primes $p\gg 0$ above which $A$ has good ordinary reduction,
\begin{equation}\label{rank zero equality}X(K_{\op{cyc}}^{(p)})=X(K).\end{equation}
Recall that $\Sigma=\Sigma_A$ is the finite set of primes $\ell$ such that $\bar{G}_\ell$ does \emph{not} contain $\op{Sp}_{2g}(\F_\ell)$, and for $\ell\in \Sigma$, we have set $K_{A,\ell}:=K(A[\ell])$. We have that for all non-anomalous primes $p$ such that 
\begin{enumerate}
    \item $A(K)[p^\infty]=0$
    \item $p>\prod_{\ell\in \Sigma} \#\left( A(K_{A,\ell})[\ell^\infty]\right)$,
    \item $\Sha(A/K)[p^\infty]=0$
    \item $p\nmid \prod_{v\nmid p} c_v(A/K)$.
\end{enumerate} the above equality \eqref{rank zero equality} holds.
\end{Th}
\begin{proof}
Let $p\geq 5$ be a prime above which $A$ has good ordinary reduction. According to Theorem \ref{cotorsion zero}, the Selmer group $\op{Sel}_{p^\infty}(A/K_{\op{cyc}}^{(p)})$ is cotorsion over $\Zx$. Assuming that $p\gg 0$ is non-anomalous, it follows from Lemma \ref{lambda is zero} that $n_0(p)=0$. Theorem \ref{th 2} then asserts that 
\[X(K_{\op{cyc}}^{(p)})=X(K),\] for all primes $p$ such that 
\begin{enumerate}
    \item $A(K)[p^\infty]=0$
    \item $p>\prod_{\ell\in \Sigma} \#\left( A(K_{A,\ell})[\ell^\infty]\right)$.
\end{enumerate}
\end{proof}

\begin{Remark}
Note that when the rank of $A(K_{\op{cyc}}^{(p)})$ is zero, then the Manin-Mumford conjecture implies that $X(K_{\op{cyc}}^{(p)})$ is finite. More generally, the conjecture asserts that $X(\bar{K})\cap A(\bar{K})_{\op{tors}}$ is finite. This result was proved by Raynaud in \cite{raynaud1983courbes}.
\end{Remark}

\subsection{An example}\label{example} We consider an explicit example. Pick the first curve in the list of \href{https://www.lmfdb.org/Genus2Curve/Q/?search_type=List&all=1}{genus 2 curves on LMFDB}, such that $\op{End}(A)=\Z$. This curve is given the equation 
\[X: y^2+(x^3+1)y=x^2+x.\] The Jacobian is an abelian surface of \emph{paramodular type}. It follows from Theorem \ref{serre pink} that $\widehat{\rho}$ has big image. For this abelian surface, $A(\Q)=\Z/14 \Z$. The analytic order of $\Sha(A/\Q)$ is approximated to $1$, therefore, it makes sense to assume that $\Sha(A/\Q)=0$. The Tamagawa product is equal to $1$.
\par According to Theorem \ref{generalization of Imai}, if $p$ is \emph{any} prime other than $2$ or $7$, $K$ is \emph{any} number field such that $A(K)[p]\neq 0$, and $K_\infty$ is \emph{any} pro-$p$ extension of $K$, then, the torsion subgroup of $A(K_\infty)$ is finite. For instance, given $p\notin \{2,7\}$, there are only finitely many quadratic extensions $K/\Q$ in which $A(K)[p]\neq 0$. This is because there are only finitely many quadratic extensions contained in the number field $\Q(A[p])$. More generally, if $K$ does not intersect $\Q(A[p])$, then the image of the residual representation does not change after base change to $K$. For any imaginary quadratic field $K$ such that $A(K)[p]=0$, one can consider the compositum of all $\Z_p$-extensions of $K$. This extension $K_\infty$ has Galois group $\op{Gal}(K_\infty/K)\simeq \Z_p^2$. Theorem \ref{generalization of Imai} applies to show that $A(K_\infty)_{\op{tors}}$ is finite.
\par According to Theorem \ref{main 2}, we have that for all non-anomalous primes $p\gg 0$ above which $A$ has good ordinary reduction, \begin{equation}\label{last eqn}X(\Q_{\op{cyc}}^{(p)})=X(\Q).\end{equation}

\bibliographystyle{abbrv}
\bibliography{references}

\end{document}